\documentclass{amsart}
\usepackage{amsmath, amsthm, amssymb,enumerate}
\usepackage[initials]{amsrefs}
\author{Joseph Flenner}
\address{University of Saint Francis, 2701 Spring Street, Fort Wayne, IN 46808, U.S.A.}
\email{jflenner@sf.edu}
\author{Vincent Guingona}
\address{University of Notre Dame, Department of Mathematics, 255 Hurley Hall, Notre Dame, IN 46556, U.S.A.}
\email{guingona.1@nd.edu}
\urladdr{http://www.nd.edu/~vguingon/}
\thanks{Both authors were supported by NSF grant DMS-0838506.}
\title{Convexly orderable groups and valued fields}
\date{\today}

\newtheorem{thm}{Theorem}[section]
\newtheorem{cor}[thm]{Corollary}
\newtheorem{lem}[thm]{Lemma}
\newtheorem{prop}[thm]{Proposition}

\newtheorem{ques}[thm]{Open Question}

\theoremstyle{remark}
\newtheorem{rem}[thm]{Remark}
\newtheorem{expl}[thm]{Example}

\theoremstyle{definition}
\newtheorem{defn}[thm]{Definition}

\newcommand{\Th}{\mathrm{Th} }

\newcommand{\fM}{\mathfrak{M}}
\newcommand{\fN}{\mathfrak{N}}

\newcommand{\cB}{\mathcal{B}}
\newcommand{\cP}{\mathcal{P}}
\newcommand{\cL}{\mathcal{L}}

\renewcommand{\phi}{\varphi}
\newcommand{\ep}{\varepsilon}
\newcommand{\fA}{\mathfrak{A}}
\newcommand{\fB}{\mathfrak{B}}
\newcommand{\fG}{\mathfrak{G}}
\newcommand{\fF}{\mathfrak{F}}
\newcommand{\set}[2]{\left\{ #1\ \middle|\  #2\right\}}
\newcommand{\ex}[2]{\exists #1 \left( #2 \right)}

\newcommand{\PP}{\mathrm{PP}}
\newcommand{\PPt}{\widetilde{\mathrm{PP}}}


\begin{document}

\begin{abstract}
 We consider the model theoretic notion of convex orderability, which fits strictly between the notions of VC-minimality and dp-minimality. In some classes of algebraic theories, however, we show that convex orderability and VC-minimality are equivalent, and use this to give a complete classification of VC-minimal theories of ordered groups and abelian groups.  Consequences for fields are also considered, including a necessary condition for a theory of valued fields to be quasi-VC-minimal. For example, the $p$-adics are not quasi-VC-minimal.
\end{abstract}

\maketitle

\section{Introduction}

After many of the advancements in modern stability theory, some model theorists have been seeking to adapt techniques from stable model theory to other families of unstable, yet still well-behaved theories. These include o-minimal theories as well as theories without the independence property. 
As these notions of model-theoretic tameness proliferate, in each case, two natural questions arise: what are the useful consequences of the property, and which interesting theories have the property? As an example of the latter line of inquiry, an ordered group is weakly o-minimal if and only if it is abelian and divisible, and an ordered field is weakly o-minimal if and only if it is real closed \cite{mms}.  Similar characterizations of dp-minimality for abelian groups can be found in \cite{ADHMS}, and results on dp-minimal ordered groups can be found in \cite{sim}.

Resting comfortably among these conditions is VC-minimality, introduced by Adler in \cite{adl}. Most of the classical variations on minimality, such as (weak) o-minimality, strong minimality, and C-minimality, imply VC-minimality. On the other hand, VC-minimality is strong enough to imply many properties of recent interest, such as dependence and dp-minimality.

The question of consequences of VC-minimality has been addressed elsewhere (see e.g.~\cites{cs,fg,gl}). In this paper, we seek to identify the VC-minimal theories among some basic classes of algebraic structures. Here a problem quickly arises. While it tends to be straightforward to verify that a theory is VC-minimal, the definition of VC-minimality does not lend itself easily to negative results. Except in some special cases, previously it had only been possible to show a theory is not VC-minimal by showing that it is not dp-minimal or dependent.

To sidestep this problem, we explore the intermediate notion of convex orderability, first introduced in \cite{gl}. All VC-minimal theories are also convexly orderable, and while the converse fails in general, in many cases it is, in a sense, close enough. The strategy, thus, is twofold. Given a class of algebraic theories, we use known results (for example, on o-minimal ordered groups) to produce a list of VC-minimal theories from the class. We then study convex orderability in relation to the class of theories to establish that the list is exhaustive.

In this way, we give a complete classification of VC-minimal theories of ordered groups (Section \ref{ordgpssect}) and abelian groups (Section \ref{abgpssect}). Partial results, in the form of necessary conditions for VC-minimality, are given for ordered fields (Section \ref{ordgpssect}) and valued fields (Section \ref{valfieldsect}). For valued fields, the weaker condition of quasi-VC-minimality is also evaluated.

The remainder of this section gives the necessary background on VC-minimality, and Section \ref{cosect} presents some useful facts about convex orderability.

\subsection{VC-minimality}

Let $X$ be any set and let $\cB \subseteq \cP(X)$.  We say that $\cB$ is \emph{directed} if, for all $A, B \in \cB$, one of the following conditions holds:
\begin{enumerate}
 \item $A \subseteq B$,
 \item $B \subseteq A$, or
 \item $A \cap B = \emptyset$.
\end{enumerate}
Let $T$ be a first-order $\cL$-theory, and fix a set of formulas
\[
 \Psi = \set{\psi_i(x; \bar{y}_i)}{i \in I}
\]
(note that the singleton $x$ is a free variable in every formula of $\Psi$, but the parameter variables $\bar{y}_i$ may vary).  Then $\Psi$ is \emph{directed} if, for all $\fM \models T$, 
\[
 \set{\psi_i(\fM; \bar{a})}{i \in I, \bar{a} \in M^{|\bar{y}_i|}}
\]
is directed, where $\psi_i(\fM; \bar{a}) = \set{b \in M}{\fM \models \psi_i(b; \bar{a})} \subseteq M$.  

We say that $T$ is \emph{VC-minimal} if there exists a directed $\Psi$ such that all (parameter-definable) formulas $\phi(x)$ are $T$-equivalent to a boolean combination of instances of formulas from $\Psi$ (i.e., formulas of the form $\psi(x; \bar{a})$ for $\psi \in \Psi$).  In this case, $\Psi$ is called a \emph{generating family} for $T$.

For example, it is easy to see that strongly minimal theories are VC-minimal; take $\Psi = \{ x = y \}$.  Similarly, o-minimal theories are VC-minimal; take $\Psi = \{ x \le y, x = y \}$.  A prototypical example of a VC-minimal theory which is neither stable nor o-minimal is the theory of algebraically closed valued fields; take $\Psi = \{ v(z)<v(x-y), v(z)\le v(x-y) \}$, recalling the swiss cheese decomposition of Holly \cite{hol1}. By a simple type-counting argument, one can see that formulas $\phi(x; \bar{y})$ in VC-minimal theories have VC-density $\le 1$ (see \cite{ADHMS}).  From this, one can conclude that VC-minimal theories are dp-minimal (see, for instance, \cite{dgl}). 

Finally, $T$ is \emph{quasi-VC-minimal} if there exists a directed $\Psi$ such that all formulas $\phi(x)$ are $T$-equivalent to a boolean combination of instances of formulas from $\Psi$ and parameter-free formulas. Clearly, all VC-minimal theories are quasi-VC-minimal.  Moreover, the theory of Presburger arithmetic, $\Th(\mathbb{Z}; +, \le)$, is quasi-VC-minimal; take $\Psi = \{ x \le y, x = y \}$.  Again, by the same type-counting argument, one can check that quasi-VC-minimal theories are dp-minimal.

\section{Convex orderability}\label{cosect}

VC-minimality is a powerful condition having many consequences (see, for example, \cites{adl,cs,fg,gl}).  However, it can be difficult to verify that a theory is not VC-minimal. In attempting to classify VC-minimal theories of certain kinds, therefore, we instead look at a related notion called convex orderability.

\begin{defn}\label{Defn_CO}
An $\cL$-structure $\fM$ is \emph{convexly orderable} if there exists a linear order $\unlhd$ on $M$ (not necessarily definable) such that, for all $\phi(x; \bar{y})$, there exists $k < \omega$ such that, for all $\overline{b} \in M^{|\bar{y}|}$, $\phi(\fM; \overline{b})$ is a union of at most $k$ $\unlhd$-convex subsets of $M$.
\end{defn}

Note in the above that $k$ may depend on $\phi$, but $\unlhd$ does not. In \cite{gl}, it is shown that if $\fM$ is convexly orderable and $\fM \equiv \fN$, then $\fN$ is convexly orderable as well.  Therefore, convex orderability is a property of a theory. Moreover, the next proposition follows immediately from the definition.

\begin{prop}\label{CO_reducts}
The property of convex orderability is closed under reducts. That is, if $T$ is a convexly orderable $\cL$-theory and $\cL'\subseteq \cL$, then the reduct $T\restriction \cL'$ is also convexly orderable.
\end{prop}

For later reference, we cite the following from \cite{gl}.

\begin{prop}[Corollary 2.9 of \cite{gl}]\label{COdp}
If $T$ is convexly orderable, then $T$ is dp-minimal.
\end{prop}

Furthermore, the following proposition is a simple modification of Proposition 2.5 of \cite{gl}.

\begin{prop}\label{Thm_ConvexOrderability01}
 Suppose $X$ is a set and $\cB \subseteq \cP(X)\setminus\{\emptyset\}$ is directed.  Then, there exists a linear ordering $\unlhd$ on $X$ so that every $B \in \cB$ is a $\unlhd$-convex subset of $X$.
\end{prop}

From this, a simple compactness argument gives the corollary.

\begin{cor}[Theorem 2.4 of \cite{gl}]\label{Cor_VCMinCO}
 If $T$ is VC-minimal and $\fM \models T$, then $\fM$ is convexly orderable.
\end{cor}

By contrast, the above corollary does not hold for quasi-VC-minimal theories, as the $\emptyset$-definable sets may be quite complicated. However, restricting our attention to a single formula, we obtain a localized result for quasi-VC-minimal theories. In the following, notice that $\unlhd$ \emph{does} depend on the formula $\phi$.

\begin{cor}\label{Cor_QuasiVCMinOrder}
 If $T$ is a quasi-VC-minimal theory, $\fM \models T$, and $\phi(x; \overline{y})$ is a formula, then there exists a linear ordering $\unlhd$ on $M$ and $k < \omega$ such that, for all $\overline{b} \in M^{|\overline{y}|}$, $\phi(\fM; \overline{b})$ is a union of at most $k$ $\unlhd$-convex subsets of $M$.  That is, $T$ is `locally convexly orderable'.
\end{cor}

\begin{proof}
 By compactness, there exists $k_0 < \omega$, $\delta(x; \overline{z})$ a directed formula, and a $\emptyset$-definable partition of $M$ via the finite set of formulas $\Theta(x)$ so that, for each $\overline{b} \in M^{|\overline{y}|}$, $\phi(\fM; \overline{b})$ is a boolean combination of at most $k_0$ instances of $\delta$ and formulas from $\Theta$. (More precisely, compactness yields $k_0$ and a finite set of formulas, while coding tricks allow one to compress a finite set of directed formulas into the single formula $\delta$.)  

Let $k = k_0 | \Theta | + 1$ and, for each $\theta \in \Theta$, let $\delta_\theta(x; \overline{z})$ be the formula $\delta(x; \overline{z}) \wedge \theta(x)$. Note that each $\delta_\theta$ is directed, as $\delta$ is.  Hence, by Theorem \ref{Thm_ConvexOrderability01}, for each $\theta \in \Theta$, there exists $\unlhd_\theta$ a linear ordering on $\theta(\fM)$ so that every instance of $\delta_\theta$ is $\unlhd_\theta$-convex.  We then concatenate the orderings $\unlhd_\theta$ in an arbitrary (but fixed) sequence to form a single linear ordering $\unlhd$ on $M$.  

Now, for any $\overline{b} \in M^{|\overline{y}|}$ and $\theta \in \Theta$, $\phi(x; \overline{b}) \wedge \theta(x)$ is a boolean combination of at most $k_0$ instances of $\delta_\theta$, each of which is $\unlhd$-convex.  Therefore, $\phi(\fM; \overline{b})$ is a union of at most $k = k_0 | \Theta | + 1$ $\unlhd$-convex subsets of $M$.
\end{proof}

One of the original motives for defining convex orderability was to give an analog to VC-minimality which is closed under reducts. However, the converse to Corollary \ref{Cor_VCMinCO} does not hold. The dense circle order is convexly orderable but not VC-minimal (for more information, see \cite{adl}). It is, in fact, a reduct of (a definitional expansion of) dense linear orders without endpoints, which is o-minimal and hence VC-minimal. On the other hand, the dense circle order becomes VC-minimal if one allows a single parameter in the generating family.

Let us call a theory \emph{VC-minimal with parameters} if there exists a directed generating family as in the original definition, but allowing parameters from some distinguished model in the formulas. One could then ask whether VC-minimality with parameters is closed under reducts. An example in \cite{acfm} shows that this is still not the case. Recalling Proposition \ref{CO_reducts}, therefore, there are convexly orderable theories which are not VC-minimal even with parameters.

Nevertheless, in the following sections we will see several instances where convex orderability serves as a useful proxy for VC-minimality. In particular, we use Corollaries \ref{Cor_VCMinCO} and \ref{Cor_QuasiVCMinOrder} to answer questions about which algebraic structures of various kinds are convexly orderable, VC-minimal, and quasi-VC-minimal.

\section{Ordered groups}\label{ordgpssect}

Let $\fG = (G; \cdot, \le)$ be an infinite ordered group and let $T = \Th(\fG)$.  We prove the following theorem.

\begin{thm}\label{Thm_VCMiniffDivisible}
 The following are equivalent:
 \begin{enumerate}
  \item $\fG$ is abelian and divisible.
  \item $T$ is o-minimal,
  \item $T$ is VC-minimal,
  \item $T$ is convexly orderable.
 \end{enumerate}
\end{thm}

This is a generalization of Theorem 5.1 of \cite{mms}, which is itself a generalization of Theorem 2.1 of \cite{pilst}.  The implications (1) $\Rightarrow$ (2) $\Rightarrow$ (3) $\Rightarrow$ (4) are well-known (or clear from the previous section), so it will suffice to show that (4) $\Rightarrow$ (1).  

Thus, suppose that $T$ is convexly orderable. By Proposition 3.3 of \cite{sim}, all dp-minimal ordered groups are abelian. Using Proposition \ref{COdp}, therefore, we already have that $\fG$ is abelian and it remains only to show that it is divisible. We begin with a general lemma about convexly orderable ordered structures.

\begin{lem}\label{Thm_CofinalDisjointSets}
 If $\fM = (M; \le, ...)$ is a linearly ordered structure that is convexly orderable, then there do not exist definable sets $X_0, X_1, ... \subseteq M$ that are pairwise disjoint and coterminal (that is, cofinal or coinitial) in $M$.
\end{lem}

\begin{proof}
Suppose that $\fM$ is convexly ordered by $\unlhd$.  Suppose that there exists definable sets $X_0, X_1, ... \subseteq M$ that are pairwise disjoint and $\le$-coterminal in $M$.  By the pigeonhole principle, we may assume that all $X_i$ are either $\le$-cofinal or $\le$-coinitial in $M$.  Without loss of generality, suppose all are $\le$-cofinal in $M$.  By convex orderability, for each $i$, $X_i$ is a union of finitely many $\unlhd$-convex subsets of $M$.  Therefore, there exists some $\unlhd$-convex subset $C_i \subseteq X_i$ such that $C_i$ is $\le$-cofinal in $M$.

Because the rays $[a,\infty)_\le$ are uniformly definable, there is a natural number $k$ such that every $[a,\infty)_\le$ is the union of at most $k$ $\unlhd$-convex sets.  Now consider the sets $C_1,\ldots,C_{2k+1}$. Since these are $\unlhd$-convex and pairwise disjoint, we may arrange the indices so that
\[
C_{i_1}\lhd C_{i_2}\lhd\ldots\lhd C_{i_{2k+1}}.
\]

For each $j\le 2k+1$, choose $b_j \in C_{i_j}$, and fix $a > \max \set{b_j}{1\le j\le 2k+1}$. By $\le$-cofinality of $C_{i_j}$, for each $j$ we may also choose $c_j \in C_{i_j} \cap [a,\infty)_\le$. Thus we have
\[
c_1 \lhd b_2 \lhd c_3 \lhd \ldots \lhd b_{2k} \lhd c_{2k+1}
\]
with each $c_j \in [a,\infty)_\le$ and each $b_j \notin [a,\infty)_\le$. It follows that for $j=0,\ldots,k$, each $c_{2j+1}$ lies in a separate $\unlhd$-convex component of $[a,\infty)_\le$. This contradiction implies that $\fM$ is not convexly orderable, as required.
\end{proof}

We return to the case of $T = \Th(\fG)$, where $\fG=(G;+,\le)$ is a convexly orderable ordered group. For $k < \omega$, let $k \mid x$ be the formula $\ex{y}{k\cdot y=x}$. For each natural number $n\ge 1$ and prime $p$, define the set
\[
D_{p,n}=\set{x\in G}{x>0, p^n\mid x \text{ and } p^{n+1}\nmid x}.
\]

\begin{lem}\label{Lem_Cofinal}
Suppose for some prime $p$ that $pG\ne G$. Then for each $n$, $D_{p,n}$ is cofinal in $G$.
\end{lem}

\begin{proof}
Since $pG\ne G$, there is some $c>0$ with $p\nmid c$. Consider $0<a\in G$. We show that there is $x\ge a$ such that $x\in D_{p,n}$. First, if $p\nmid a$, let $b=a$; if $p\mid a$, set $b=a+c$. So, $b\ge a$ and $p\nmid b$. Now $x=p^n\cdot b\ge a$ and $x\in D_{p,n}$.
\end{proof}

Combining this with Lemma \ref{Thm_CofinalDisjointSets}, we can now easily establish Theorem \ref{Thm_VCMiniffDivisible}.

\begin{cor}\label{VC-DOAG}
 If $\fG$ is convexly orderable, then $\fG$ is divisible.
\end{cor}

\begin{proof}
 Suppose $\fG$ is convexly orderable but not divisible, say $pG \ne G$. For each $n$, $D_{p,n}$ is cofinal and pairwise disjoint in $\fG$.  Apply Lemma \ref{Thm_CofinalDisjointSets} to conclude.
\end{proof}

Although there were previously known examples of dp-minimal theories that are not VC-minimal (e.g., see \cite{dgl}), this gives us a natural example of such a theory (discovered independently in \cite{acfm}).

\begin{expl}\label{Cor_StuffnotVCMin}
 The theory of Presburger arithmetic, $T = \Th(\mathbb{Z}; +, \le)$, is not VC-minimal and not convexly orderable.  On the other hand, it is quasi-VC-minimal, and hence also dp-minimal.
\end{expl}

This has interesting consequences for ordered fields.

\begin{prop}\label{Prop_OrderedFieldVCMin}
 Suppose $\fF = (F; +, \cdot, \le)$ is an ordered field.  If $\fF$ is convexly orderable, then every positive element has an $n^\text{th}$ root for all $n \ge 1$.
\end{prop}

\begin{proof}
 Suppose $\fF$ is convexly ordered by $\unlhd$.  Then, $\unlhd$ induces a convex ordering on the ordered group $(F_+; \cdot, \le)$ where $F_+ = \set{a \in F}{a > 0}$. Thus, by Theorem \ref{Thm_VCMiniffDivisible}, $F_+$ is divisible.  In other words, for any $a \in F_+$ and $n \ge 1$, there exists $b \in F_+$ such that $b^n = a$.
\end{proof}

Theorem 5.3 of \cite{mms} states that any weakly o-minimal ordered field is real closed.  This suggests the following open question.

\begin{ques}\label{Ques_RealClosed}
 Is it the case that an ordered field $(F; +, \cdot, \le)$ is convexly orderable if and only if $(F; +, \cdot, \le)$ is real closed?
\end{ques}

Before we get carried away, however, not all ordered structures that are convexly orderable are weakly o-minimal. For example, consider $\mathbb{Q}$ and take $D \subseteq \mathbb{Q}$ dense and codense. One can verify that the structure $\fM = ( \mathbb{Q}; \le, D)$ has quantifier elimination, from which it easily follows that it is VC-minimal. For instance, take as a generating family
\[
 \Psi = \{ (D(x) \wedge x < y), (\neg D(x) \wedge x < y), D(x), x = y \}.
\]
So $\fM$ is convexly orderable, but on the other hand, $\fM$ is clearly not weakly o-minimal.  The issue is that Lemma \ref{Thm_CofinalDisjointSets} necessitates \emph{infinitely many} coterminal disjoint sets to contradict convex orderability.  This leads to another open question.

\begin{ques}\label{Ques_QuasiWeaklyOMin}
 If $\fM = (M; \le, ...)$ is a linearly ordered structure that is convexly orderable, then is $\fM$ quasi-weakly o-minimal?
\end{ques}

\section{Valued fields}\label{valfieldsect}

\subsection{Simple interpretability}

In this subsection we exhibit a means of passing convex orderability from a structure to a simple interpretation in the structure. If $\fM$ and $\fN$ are models (not necessarily in the same language) and $A \subseteq M$, then $\fM$ \emph{interprets $\fN$ over $A$} if there are $n\ge 1$, an $A$-definable subset $S \subseteq M^n$, and an $A$-definable equivalence relation $\ep$ on $S$ such that
\begin{itemize}
\item
the elements of $\fN$ are in bijection with the $\ep$-equivalence classes of $S$, and
\item
the relations on $S$ induced by the relations and functions of $\fN$ via this bijection are $A$-definable in $\fM.$
\end{itemize}
Moreover, if $n=1$ in the above definition, we say that $\fM$ \emph{simply interprets} $\fN$.

It is generally most convenient to identify the elements of $\fN$ with the equivalence classes of $S$, so that for instance we will write $\bar{a} \in x$ if $\bar{a} \in S$ and $x \in N$ corresponds to the $\ep$-equivalence class containing $\bar{a}$.

\begin{rem}
\label{r21}
Using the same notation as above, suppose $\phi(\bar{x};\bar{y})$ is a formula in the language of $\fN$ with $k=|\bar{x}|$. Then there is $\tilde{\phi}(\bar{z};\bar{w})$ in the language of $\fM$ (with parameters from $A$) with the property that, for any set $X\subseteq N^k$ defined by an instance $\phi(\bar{x};\bar{a})$ of $\phi$, the set
\[
\tilde{X}=\bigcup X\subseteq S^k.
\]
is defined by an instance $\tilde{\phi}(\bar{z};\bar{b})$ of $\tilde{\phi}$.
To see this, induct on the complexity of $\phi$, replacing function and relation symbols from $\fN$ with their corresponding definitions in $\fM$ and $=$ with $\ep$, and relativizing all quantifiers to $S$.
\end{rem}

\begin{lem}
\label{l22}
If $\fM$ simply interprets $\fN$ and $\fM$ is convexly orderable, then $\fN$ is also convexly orderable.
\end{lem}

\begin{proof}
Let $\ep(x,y)$ define an equivalence relation on $S \subseteq M$ as in the definition of interpretation (possibly over parameters), and suppose that $\fM$ is convexly ordered by $\unlhd_M$. Define on $\fN$ the relation $\unlhd_N$ by
\[
x\unlhd_N y \Longleftrightarrow (\forall s\in y)(\exists r\in x)[r\unlhd_M s].
\]
We claim that $\fN$ is convexly ordered by $\unlhd_N$.

First note that $\unlhd_N$ linearly orders $N$. Transitivity and linearity are clear. For antisymmetry, suppose that $x\unlhd_N y$ and $y\unlhd_N x$. Then, beginning with an arbitrary $s_0\in y$, find $r_i\in x$, $s_i\in y$ such that for every $i<\omega$, $r_i\unlhd_M s_i$ and $s_{i+1}\unlhd_M r_i$. But since $x$ is a definable subset of $\fM$, $x$ must be a finite union of $\unlhd_M$-convex sets. So we must have $s_i\in x$ for some $i$, whence $x=y$. A similar argument shows that $x\lhd_N y$ iff there is an $r\in x$ such that $r\lhd_M s$ for all $s\in y$.

Now consider a formula $\phi(x;\bar{y})$ in the language of $\fN$, $\bar{a}$ a tuple from $N$, and $X\subseteq N$ the set defined by $\phi(x;\bar{a})$. For $\tilde{\phi}(x;\bar{b})$ defining $\tilde{X}$ as in Remark \ref{r21}, since $\unlhd_M$ convexly orders $\fM$, there is a uniform bound $k$ on the number of $\unlhd_M$-convex sets comprising an instance of $\tilde{\phi}$ in $\fM$. It will suffice to show that $X$ is also a union of at most $k$ $\unlhd_N$-convex sets in $N$.

Suppose not, so that there are 
\[c_0\unlhd_N c_1\unlhd_N\ldots\unlhd_N c_{2k}\]
such that $c_i\in X$ iff $i$ is even. For each $i<2k$, since $c_i\neq c_{i+1}$ there is $\tilde{c}_i\in c_i$ such that $\tilde{c}_i\lhd_M d$ for all $d\in c_{i+1}$. Take also any element $\tilde{c}_{2k}\in c_{2k}$. Now
\[\tilde{c}_0\unlhd_M \tilde{c}_1\unlhd_M\ldots\unlhd_M \tilde{c}_{2k}\]
and $\tilde{c}_i\in\tilde{X}$ iff $i$ is even. This contradicts the fact that $\tilde{X}$ is a union of $k$ (or fewer) $\unlhd_M$-convex sets.

We conclude that in $\fN$, every instance of $\phi$ defines a union of $k$ or fewer $\unlhd_N$-convex sets. Since any formula in the language of $\fN$ admits such a uniform bound, $\unlhd_N$ convexly orders $\fN$.
\end{proof}

Lemma \ref{l22} allows us to show that a theory is not convexly orderable (hence not VC-minimal) by simply interpreting a structure that is not convexly orderable.  We can apply this to theories of valued fields.  Let $K$ be a valued field with value group $\Gamma$, residue field $k$, and valuation $v: K \rightarrow \Gamma \cup \{ \infty \}$, and let $T = \Th(K; +, \cdot, |)$.  Here $x | y$ means $v(x) \le v(y)$.  Though we work in the one-sorted language $\cL = \{ +, \cdot, | \}$, the statements could be adapted to other languages of valued fields.

\begin{cor}
 \label{c23}
 If $T$ is convexly orderable, then both the value group $\Gamma$ and the residue field $k$ are convexly orderable.
\end{cor}

\begin{proof}
 Both $\Gamma$ and $k$ are simply interpretable (over $\emptyset$) in $K$. For example, $\Gamma$ is interpreted on $S=K\setminus\left\{0\right\}$ via $\ep(x,y)\equiv x\mid y\wedge y\mid x$ (i.e., $v(x)=v(y)$). Since $v(xy)=v(x)+v(y)$, the addition in $\Gamma$ is interpreted by multiplication in $K$, and the ordering is explicitly given by $\mid$.  We use Lemma \ref{l22} to conclude.
\end{proof}

We know that the theory of algebraically closed valued fields is convexly orderable.  Also, the theory of real closed valued fields is weakly o-minimal \cite{mad}, hence also convexly orderable.  This leads to an interesting open question: Under which circumstances does the converse of Corollary \ref{c23} hold?

\begin{ques}\label{Ques_ValuedFields}
 Is it true that, for any Henselian valued field $K$ with value group $\Gamma$ and residue field $k$, $K$ is convexly orderable if and only if $\Gamma$ and $k$ are convexly orderable?
\end{ques}

We understand when $\Gamma$ is convexly orderable by Theorem \ref{Thm_VCMiniffDivisible}, but we do not currently have a characterization for when $k$ is convexly orderable.  Answering Open Question \ref{Ques_ValuedFields} would probably require first understanding when a field is convexly orderable in general.

We can apply Corollary \ref{c23} to the case of the $p$-adics.

\begin{cor}\label{Cor_NonDivisibleVCMin}
 If $\Gamma$ is not divisible, then $T$ is not convexly orderable, hence not VC-minimal.  In particular, the theory of the $p$-adics is not VC-minimal.
\end{cor}

\begin{proof}
 By Theorem \ref{Thm_VCMiniffDivisible}, $\Gamma$ is convexly orderable if and only if $\Gamma$ is divisible.  Hence, if $\Gamma$ is not divisible, then Corollary \ref{c23} implies that $T$ is not convexly orderable.  In particular, the theory of the $p$-adics, $\Th(\mathbb{Q}_p; +, \cdot, |)$, has value group $(\mathbb{Z}; +, \le)$, which is not divisible.  Hence, the theory of the $p$-adics is not VC-minimal.
\end{proof}

By Section 6 of \cite{dgl}, the theory of the $p$-adics is dp-minimal.  So this corollary gives us another natural example of a theory that is dp-minimal but not VC-minimal.  In the next subsection, we exhibit a means of producing examples of theories that are dp-minimal but not quasi-VC-minimal.

\subsection{Quasi-VC-minimality}

For this subsection, fix $K$ a valued field with value group $\Gamma$ and let $T = \Th(K; +, \cdot, |)$ as in the previous subsection. First, recall that if $K$ is algebraically closed, then $T$ is VC-minimal. Notice that if $K$ is algebraically closed, then $\Gamma$ is divisible.  The main goal of this section is to prove the following stronger result.

\begin{thm}\label{thm_VFNotQVCMin}
 If $T$ is quasi-VC-minimal, then $\Gamma$ is divisible.
\end{thm}

Suppose then that $\Gamma$ is not divisible, say $p\Gamma\ne\Gamma$.  Fix some positive $\gamma_1 \in \Gamma\setminus p\Gamma$.  Define $\gamma_n \in \Gamma$ by
\[
 \gamma_n =
 \begin{cases}
  k \cdot p \cdot \gamma_1 & \text{if } n = 2k, \\
  \gamma_1 + k \cdot p \cdot \gamma_1 & \text{if } n = 2k+1.
 \end{cases}
\]
Notice that $0 = \gamma_0 < \gamma_1 < ... < \gamma_n < ...$ and $p\mid\gamma_n$ if and only if $n$ is even.

We now construct, for each $n < \omega$, $\mathcal{A}_n \subseteq K$ as follows.  Set $\mathcal{A}_0 = \{ 0 \}$.  For each $a \in \mathcal{A}_n$, choose $a' \in K$ such that $v(a - a') = \gamma_n$.  Let
\[
 \mathcal{A}_{n+1} = \mathcal{A}_n \cup \set{a'}{a \in \mathcal{A}_n}.
\]
Note that $a'\notin \mathcal{A}_n$ (to see this, show inductively that for distinct $b_1,b_2\in\mathcal{A}_n$, $v(b_1-b_2)\le\gamma_{n-1}$). Therefore $| \mathcal{A}_n | = 2^n$.  Moreover, for all $a \in \mathcal{A}_n$ and all $i < n$, there exists $b \in \mathcal{A}_n$ such that $v(a - b) = \gamma_i$.

Suppose that $\unlhd$ is a linear ordering on $K$.  In this case, each $\mathcal{A}_n$ is also linearly ordered by $\unlhd$.  For each $b \in K$, define
\[
 X_b = \set{a \in K}{p \mid v(a - b)}.
\]

\begin{lem}\label{Lem_NotFinitelyMany}
 For each $n < \omega$, there exists $b \in K$ such that $X_b$ is the union of no fewer than $n+1$ $\unlhd$-convex subsets of $K$.
\end{lem}

\begin{proof}
 Fix $n < \omega$ and let $\mathcal{A} = \mathcal{A}_{2n+1}$, which is a finite linear order (under $\unlhd$).
 
 Let $a_0 \in \mathcal{A}$ be the $\unlhd$-minimal element.  In general, we inductively construct a sequence $a_0, ..., a_{2n+1} \in \mathcal{A}$ such that
 \begin{enumerate}
  \item $v(a_j - a_i) = \gamma_j$ for all $j < i$,
  \item $a_0 \lhd a_1 \lhd ... \lhd a_{2n+1}$, and
  \item for all $a \in \mathcal{A}$ with $v(a - a_i) \ge \gamma_i$, $a_i \unlhd a$.
 \end{enumerate}
 Suppose that $a_0,\ldots,a_i$ with the above properties have been found, and choose $a_{i+1} \in \mathcal{A}$ $\unlhd$-minimal such that $v(a_{i+1} - a_i) = \gamma_i$.  This exists by definition of $\mathcal{A} = \mathcal{A}_{2n+1}$.  By condition (3), $a_i \lhd a_{i+1}$, so condition (2) holds up to $a_{i+1}$.  Condition (1) and $v(a_{i+1} - a_i) = \gamma_i>\gamma_j$ implies that $v(a_j - a_{i+1}) = \gamma_j$ for all $j < i$.  Therefore, condition (1) holds for $a_{i+1}$.  Finally, fix $a \in \mathcal{A}$ and suppose $v(a - a_{i+1}) \ge \gamma_{i+1}$.  Since $v(a_{i+1} - a_i) = \gamma_i$, we have $v(a - a_i) = \gamma_i$ as well.  However, since $a_{i+1}$ was chosen $\unlhd$-minimal in the set $\set{x \in \mathcal{A}}{v(x - a_i) = \gamma_i}$ and $a$ belongs to this set, we must have that $a_{i+1} \unlhd a$.  Thus, condition (3) holds for $a_{i+1}$.
  
 Finally, set $b = a_{2n+1}$.  Then, for $i \le 2n$, $a_i \in X_b$ if and only if $p \mid v(a_i - b)$ if and only if $p \mid \gamma_i$.  Recall, moreover, that $p \mid \gamma_i$ if and only if $i$ is even.  Therefore, $a_i \in X_b$ if and only if $i$ is even.  By condition (2), $X_b$ is the union of no fewer than $n+1$ $\unlhd$-convex subsets of $K$. 
\end{proof}

\begin{proof}[Proof of Theorem \ref{thm_VFNotQVCMin}]
 Suppose $\Gamma\ne p\Gamma$.  Fix the formula
 \[
  \phi(x;y) = \exists z ( z^p \mid (x - y) ).
 \]
 Towards a contradiction, suppose $T$ were quasi-VC-minimal.  By Corollary \ref{Cor_QuasiVCMinOrder}, there exists a linear order $\unlhd$ on $K$ and $n < \omega$ such that each instance of $\phi$ is a union of at most $n$ $\unlhd$-convex subsets of $K$.  By Lemma \ref{Lem_NotFinitelyMany}, there exists $b \in K$ such that $X_b = \phi(K; b)$ is a union of no fewer than $n+1$ $\unlhd$-convex subsets of $K$, a contradiction.
\end{proof}

\begin{cor}\label{Cor_pAdicsnotQVCMin}
 The following theories are not quasi-VC-minimal: $\Th(\mathbb{Q}_p; +, \cdot, |)$ for any prime $p$, and $\Th(k((t)); +, \cdot, |)$ for any field $k$.
\end{cor}

Since the $p$-adics are dp-minimal, this gives us a natural example of a theory that is dp-minimal and not quasi-VC-minimal.  Combining this observation with Corollary \ref{Cor_StuffnotVCMin}, we get strict implications
\[
 \text{ VC-minimal } \Rightarrow \text{ quasi-VC-minimal } \Rightarrow \text{ dp-minimal }
\]
where strictness is witnessed by Presburger arithmetic and the $p$-adics respectively.

\section{Abelian Groups}\label{abgpssect}

Let $\fA = (A; +)$ be an abelian group and $T = \Th(\fA)$. Throughout this section we work exclusively in the pure group language $\cL=\left\{+\right\}$.  For each $k, m < \omega$, consider the formula
\[
 \phi_{k,m}(x) = \ex{y}{k \cdot y = m \cdot x}.
\]
Notice that $\phi_{k,m}(\fA)$ is a subgroup of $A$.  For $k = 0$, $\phi_{0,m}(\fA)$ is the subgroup of $m$-torsion elements of $A$, which we will also denote by $A[m]$.  For $m = 1$, $\phi_{k,1}(\fA)$ is the subgroup of $k$-multiples of $A$, which we will also denote by $kA$.  

\begin{prop}[Corollary 2.13 of \cite{prest}]\label{Thm_PPElim}
 All definable subsets of $A$ are boolean combinations of cosets of $\phi_{k,m}(\fA)$ for various $k, m < \omega$.
\end{prop}

Let $\PP(A)$ be the set of all the p.p.-definable subgroups of $A$, which are namely the finite intersections of subgroups of the form $\phi_{k,m}(\fA)$ for various $k, m < \omega$.  Define a quasi-order $\precsim$ on all subgroups of $A$ by setting, for each subgroup $B_0$ and $B_1$ of $A$:
\begin{equation*}\label{Eq_QuasiorderPP}
 B_0 \precsim B_1 \text{ if and only if } [B_0 : B_0 \cap B_1] < \aleph_0.
\end{equation*}
Think of this as $B_0$ being almost a subgroup of $B_1$ (missing only by a finite index).  This quasi-order generates an equivalence relation $\sim$, which is called \textit{commensurability}.  For any $B_0 \sim B_1$, notice that $B_0 \cap B_1 \sim B_0$, so $\sim$-classes are closed under intersection.  We denote by $\PPt(A)$ the set $\PP(A) / \sim$ of equivalence classes.  Thus, $\precsim$ induces a partial order on $\PPt(A)$.  In \cite{ADHMS}, this partial order is used to characterize dp-minimality of $T$ as follows.

\begin{prop}[Corollary 4.12 of \cite{ADHMS}]\label{Prop_dpMinChar}
 The theory $T$ is dp-minimal if and only if $\left( \PPt(A); \precsim \right)$ is linear.
\end{prop}

This is then used as the main tool for proving a classification of dp-minimal theories of abelian groups. In the following, a \emph{nonsingular} group $B$ is one for which $B[p]$ and $B/pB$ are finite for all primes $p$.

\begin{prop}[Proposition 5.27 of \cite{ADHMS}]\label{Prop_ADHMS527}
 The theory $T$ is dp-minimal if and only if $\fA$ is elementarily equivalent to one of the following abelian groups:
 \begin{enumerate}
  \item 
  $\bigoplus\limits_{i\ge 1}\left(\mathbb{Z} / p^i \mathbb{Z} \right)^{(\alpha_i)} \oplus \mathbb{Z}\left( p^\infty \right)^{(\beta)} \oplus \left( \mathbb{Z}_{(p)} \right)^{(\gamma)} \oplus B$ for some prime $p$, a nonsingular abelian group $B$, and $\alpha_i$, $\beta$, and $\gamma$ cardinals with $\alpha_i < \aleph_0$ for all $i$.
  \item 
  $\left( \mathbb{Z} / p^k \mathbb{Z} \right)^{(\alpha)} \oplus \left( \mathbb{Z} / p^{k+1} \mathbb{Z} \right)^{(\beta)} \oplus B$ for some prime $p$, $k \ge 1$, finite abelian group $B$, and cardinals $\alpha$ and $\beta$, at least one of which is infinite.
 \end{enumerate}
\end{prop}

In this section, we will prove a characterization for when $T$ is VC-minimal (and convexly orderable) analogous to Proposition \ref{Prop_dpMinChar}, and likewise use it to obtain a complete list of VC-minimal theories of abelian groups.

\begin{lem}\label{Lem_VCMinAbGrps}
 Suppose that there exists $\mathcal{H} \subseteq \PP(A)$ such that
 \begin{enumerate}
  \item $(\mathcal{H}; \subseteq)$ is a linear order; and
  \item For all $k$ and $m$, $\phi_{k,m}(\fA)$ is a boolean combination of cosets of elements $H \in \mathcal{H}$.
 \end{enumerate}
 Then, $T$ is VC-minimal.
\end{lem}

\begin{proof}
  For each $H \in \mathcal{H}$, let $\psi_H(x; y)$ be the formula $x - y \in H$, and let $\Psi = \set{\psi_H}{H \in \mathcal{H}}$. The instances of $\Psi$ define precisely the cosets of members of $\mathcal{H}$. We claim that $\Psi$ is a generating family for $T$.  

First, to see that $\Psi$ is directed, fix $H_1, H_2 \in \mathcal{H}$ and $a_1, a_2 \in A$.  By (1), we may assume without loss of generality that $H_1 \subseteq H_2$.  Then each coset of $H_1$ is a subset of a coset of $H_2$, so that either $a_1 + H_1 \subseteq a_2 + H_2$ or $(a_1 + H_1) \cap (a_2 + H_2) = \emptyset$ as required.  

By Proposition \ref{Thm_PPElim}, all definable subsets of $A$ are boolean combinations of cosets of $\phi_{k,m}(\fA)$ for various $k, m < \omega$.  So (2) implies that all parameter-definable subsets of $A$ are in fact boolean combination of cosets of elements $H \in \mathcal{H}$.
\end{proof}

\begin{cor}\label{Cor_ZPlus}
 The theory $T = \Th(\mathbb{Z}; +)$ is VC-minimal.
\end{cor}

\begin{proof}
 Let $\mathcal{H} = \set{(n!) \cdot \mathbb{Z}}{1 \le n < \omega} \cup \{ 0 \}$.  This satisfies the conditions in Lemma \ref{Lem_VCMinAbGrps}.
\end{proof}

For a prime $p$, let $\mathbb{Z}_{(p)}$ be the additive group of the ring $\mathbb{Z}$ localized at the prime ideal $(p) = p \mathbb{Z}$.  Let $\mathbb{Z}(p^\infty)$ be the Pr\"{u}fer $p$-group, which is the direct limit of $(\mathbb{Z} / p^k \mathbb{Z})$ for all $k \ge 1$.  For an abelian group $A$ and cardinal $\kappa$, let $A^{(\kappa)}$ be the direct sum of $\kappa$ copies of $A$.

\begin{cor}\label{Cor_OtherVCMin}
 The theories of the following abelian groups are VC-minimal:
 \begin{enumerate}
  \item $\left( \mathbb{Z} / p^k \mathbb{Z} \right)^{(\aleph_0)}$ for some $k < \omega$ and prime $p$,
  \item $\left( \mathbb{Z} / p^k \mathbb{Z} \right)^{(\aleph_0)} \oplus \left( \mathbb{Z} / p^{k+1} \mathbb{Z} \right)^{(\aleph_0)}$ for some $k < \omega$ and prime $p$, and
  \item $\mathbb{Z}(p^\infty)^{(\beta)} \oplus \mathbb{Z}_{(p)}^{(\gamma)}$ for cardinals $\beta$ and $\gamma$ and prime $p$.
 \end{enumerate}
\end{cor}

\begin{proof}
 (1) Since $p^i A = \left( p^i \mathbb{Z} / p^k \mathbb{Z} \right)^{(\aleph_0)}$ and $A[p^i] = \left( p^{k-i} \mathbb{Z} / p^k \mathbb{Z} \right)^{(\aleph_0)}$, we see that
 \[
  \PP(A) = \set{\left( p^i \mathbb{Z} / p^k \mathbb{Z} \right)^{(\aleph_0)}}{0 \le i \le k},
 \]
 which is itself a chain.  We conclude that $T$ is VC-minimal by Lemma \ref{Lem_VCMinAbGrps}.
 
 (2) Notice that, for each $i$, 
\begin{align*}
p^i A &= \left( p^i \mathbb{Z} / p^k \mathbb{Z} \right)^{(\aleph_0)} \oplus \left( p^i \mathbb{Z} / p^{k+1} \mathbb{Z} \right)^{(\aleph_0)}\\ 
A[p^i] &= \left( p^{k-i} \mathbb{Z} / p^k \mathbb{Z} \right)^{(\aleph_0)} \oplus \left( p^{k+1-i} \mathbb{Z} / p^{k+1} \mathbb{Z} \right)^{(\aleph_0)}.
\end{align*}  
So, let $\mathcal{H}$ be the chain $0 \subset p^k A \subset A[p] \subset p^{k-1} A \subset A[p^2] \subset ...$ and use Lemma \ref{Lem_VCMinAbGrps} to conclude.
 
 (3) In this case, we have
\begin{align*}
A[p^i] &= (\mathbb{Z}(p^\infty)[p^i])^{(\beta)} \oplus 0\\ 
p^i A &= \mathbb{Z}(p^\infty)^{(\beta)} \oplus \left( p^i \mathbb{Z}_{(p)} \right)^{(\gamma)}.
\end{align*}  
Use the chain $0 \subseteq A[p] \subseteq A[p^2] \subseteq ... \subseteq p^2 A \subseteq p A \subseteq A$ along with Lemma \ref{Lem_VCMinAbGrps} to conclude.
\end{proof}

However, not every dp-minimal abelian group is VC-minimal or even convexly orderable.

\begin{lem}\label{Lem_dpMinnotVCMin}
 Suppose that there exists a chain of $\emptyset$-definable subgroups $A = A_0 \supseteq A_1 \supseteq A_2 \supseteq ...$ and a $\emptyset$-definable subgroup $B \subseteq A$ such that
 \begin{enumerate}
  \item for each $i < \omega$, $A_i \cap B \neq A_{i+1} \cap B$, and
  \item for each $i < \omega$, $[A_i : A_i \cap B] \ge \aleph_0$.
 \end{enumerate}
 Then, $T = \Th(A; +)$ is not convexly orderable.  Hence, $T$ is not VC-minimal.
\end{lem}

\begin{proof}
 By way of contradiction, suppose that $\fA$ is convexly ordered by $\unlhd$.  In particular, suppose that each instance of the formula $x - y \in B$ is a union of at most $k$ $\unlhd$-convex subsets of $A$ for some fixed $k < \omega$.  

Since $[A_k : A_k \cap B] \ge \aleph_0$, the set 
\[
\mathcal{C} = \set{a + B}{a \in A_k}
\]
 of cosets of $B$ is infinite.
 On the other hand, for each $1 \le i \le k$, $A_i \cap B \subsetneq A_{i-1} \cap B$.  So, for any choice of $b \in (A_{i-1} \setminus A_i) \cap B$ and $a \in A_k$, $a + b \in (A_{i-1} \setminus A_i)$.  Therefore, for all $a \in A_k$ and $1 \le i \le k$, $(a + B) \cap (A_{i-1} \setminus A_i)$ is non-empty.  That is, $(A_{i-1} \setminus A_i)$ intersects non-trivially each element of $\mathcal{C}$.
 
 By convex orderability, for each $i \le k$, $A_i$ is a finite union of $\unlhd$-convex subsets of $A$.  Let $\mathcal{C}_i$ denote the elements $a + B \in \mathcal{C}$ such that, for some $\unlhd$-convex component $C$ of $a + B$, $C \nsubseteq A_i$ and $C \cap A_i \neq \emptyset$.  By convexity, there can be only finitely many such $a + B$, namely the ones covering the finitely many ``endpoints'' of $A_i$.  Hence, $\mathcal{C}_i$ is finite for each $i\le k$.  Finally, set 
 \[
  \mathcal{C}^* = \mathcal{C} \setminus \left( \bigcup_{i \le k} \mathcal{C}_i \right).
 \]
 Since $\mathcal{C}$ is infinite, $\mathcal{C}^*$ is also infinite and, in particular, non-empty.  

We claim that each $A_i$ contains at most $k - i$ $\unlhd$-convex components of each element of $\mathcal{C}^*$. By choice of $k$, this clearly holds for $i=0$. So suppose that $i > 0$ and that the claim holds for $A_{i-1}$.  Consider $a+B\in\mathcal{C}^*$. By construction, for each $\unlhd$-convex component $C$ of $a + B$, either $C \subseteq A_i$ or $C \cap A_i = \emptyset$.  However, as observed above $(a + B) \cap (A_{i-1} \setminus A_i) \neq \emptyset$, so at least one of the $\unlhd$-convex components of $a + B$ contained in $A_{i-1}$ must be disjoint from $A_i$. By assumption, $A_{i-1}$ contains at most $k-(i-1)$ $\unlhd$-convex components of $a+B$. Thus $A_i$ contains at most $k-i$. The conclusion follows by induction.
 
Therefore, for all $a + B \in \mathcal{C}^*$, $(a + B) \cap A_k = \emptyset$. On the other hand, $A_k$ intersects every coset $a+B\in\mathcal{C}$ by definition of $\mathcal{C}$. This gives the desired contradiction.
\end{proof}

We use this to produce an example of an abelian group whose theory is dp-minimal but not VC-minimal.

\begin{cor}\label{Cor_ExamplesNotVCMinAG}
 Fix some $\alpha_i<\aleph_0$ for each $i \ge 1$ such that the set $\set{i}{\alpha_i > 0}$ is infinite.  Then the theory of the abelian group
 \[
  A = \bigoplus_{i \ge 1} \left( \mathbb{Z} / p^i \mathbb{Z} \right)^{\alpha_i}
 \]
is not convexly orderable.
\end{cor}

\begin{proof}
 Let $I = \set{i}{\alpha_i > 0}$, let $i_0 = 0$, and let $i_1 < i_2 < ...$ enumerate $I$.  It is straightforward to check that the $\emptyset$-definable subgroups
\[
  A_\ell = p^{i_\ell} A \text{ for all } \ell < \omega, \text{ and } B = A[p]
 \]
satisfy the hypotheses of Lemma \ref{Lem_dpMinnotVCMin}.
\end{proof}

By Proposition \ref{Prop_ADHMS527} (1), we see that this $A$ is, in fact, dp-minimal.

\begin{defn}\label{Defn_UpwardCohere}
 For $X \in \PPt(A)$ (i.e., $X$ is a $\sim$-class of $\PP(A)$), we say that $X$ is \emph{upwardly coherent} if there exists $H \in X$ such that, for all $H_1 \in \PP(A)$ with $H \precnsim H_1$, we have that $H \subseteq H_1$.

By extension, we say that the group $A$ is \emph{upwardly coherent} if every $X\in\PPt(A)$ is.
\end{defn}

Intuitively, upward coherence means the class contains a particular subgroup for which being \emph{almost} a proper subgroup is sufficient to be, in fact, a subgroup. In the presence of dp-minimality, this condition implies VC-minimality as shown in the next lemma.

\begin{lem}\label{Lem_UpwardCohereVCMin}
Suppose $T=\Th(A;+)$ is dp-minimal.  If $A$ is upwardly coherent, then $T$ is VC-minimal.
\end{lem}

\begin{proof}
 For each $X \in \PPt(A)$, let $H_X \in X$ witness that $X$ is upwardly coherent.  Since $\PP(A)$ is countable, so is $X$, so let $X = \set{H_i}{i < \omega}$ enumerate $X$.  Define $H^i_X \in X$ inductively as follows: 
\begin{itemize}
\item
$H^0_X = H_X$.
\item
For $i \ge 0$, $H^{i+1}_X = H^i_X \cap H_i$.
\end{itemize}
Since $X$ is closed under intersection, each $H^i_X$ is still an element of $X$.  

Let $\mathcal{H}_X = \set{H^i_X}{i < \omega}$.  By construction, $\mathcal{H}_X$ is a chain under $\subseteq$ with maximal element $H_X$.  Moreover, by definition of $\sim$, every $H \in X$ is a \emph{finite} union of cosets of a member of $\mathcal{H}_X$.  Finally, set
 \[
  \mathcal{H} = \bigcup \set{\mathcal{H}_X}{X \in \PPt(A)}.
 \]
 
For any distinct $X, Y \in \PPt(A)$, by Proposition \ref{Prop_dpMinChar} either $X \precnsim Y$ or $Y \precnsim X$.  Without loss, suppose $X \precnsim Y$.  Therefore, by upward coherence, $H \supseteq H_X$ for all $H\in Y$.  Hence, $\mathcal{H}_X \cup \mathcal{H}_Y$ is a chain under $\subseteq$.  It follows that $\mathcal{H}$ is itself a chain under $\subseteq$.  We thus conclude that $\mathcal{H}$ satisfies the hypotheses of Lemma \ref{Lem_VCMinAbGrps}, showing that $T$ is VC-minimal.
\end{proof}

Putting this all together, we arrive at the desired characterization of convexly orderable (and VC-minimal) abelian groups.

\begin{thm}\label{Thm_CharofVCMinAG}
 The following are equivalent:
 \begin{enumerate}
  \item $T$ is VC-minimal;
  \item $T$ is convexly orderable;
  \item $T$ is dp-minimal and $A$ is upwardly coherent.
 \end{enumerate}
\end{thm}

\begin{proof}
 We have (1) $\Rightarrow$ (2) by Corollary \ref{Cor_VCMinCO}.  Lemma \ref{Lem_UpwardCohereVCMin} gives (3) $\Rightarrow$ (1).  Thus, it remains only to show (2) $\Rightarrow$ (3).
 
If $T$ is convexly orderable, then $T$ is dp-minimal by Proposition \ref{COdp}.  So, suppose that there exists some $X \in \PPt(A)$ that is not upwardly coherent.  Fixing any $B \in X$, we construct $A = A_0 \supseteq A_1 \supseteq A_2 \supseteq ...$ from $\PP(A)$ such that, for all $i < \omega$:
 \begin{enumerate}
  \item $[B : A_i \cap B] < \aleph_0$, so that $A_i \cap B \in X$; 
  \item If $i>0$, then $A_{i-1} \cap B \neq A_i \cap B$; and
  \item $[A_i : A_i \cap B] \ge \aleph_0$.
 \end{enumerate}
By Lemma \ref{Lem_dpMinnotVCMin}, this implies that $T$ is not convexly orderable, as required.  

First, set $A_0 = A$.  If $B \sim A$, then $A \in X$ trivially witnesses upward coherence, contrary to assumption.  Therefore, $[A : B] \ge \aleph_0$, giving condition (3) for $i=0$.  Clearly condition (1) and (2) also hold for $i=0$.  

Now fix $i \ge 0$ and suppose that $A_i$ has been constructed satisfying (1), (2), and (3).  Consider $A_i \cap B$.  Since $A_i \cap B \in X$ and $X$ is not upwardly coherent, there exists $H \in \PP(A)$ such that $A_i \cap B \precnsim H$ and $A_i \cap B \nsubseteq H$.  Set $A_{i+1} = H \cap A_i$. We show that $A_{i+1}$ satisfies (1), (2), and (3).
 
 Since $A_i \cap B \precsim H$, 
\[
[A_i \cap B : A_{i+1} \cap B] = [A_i \cap B : H \cap A_i \cap B] < \aleph_0,
\] 
giving condition (1).  Suppose $A_i \cap B = A_{i+1} \cap B$.  Then $H \cap (A_i \cap B) = A_i \cap B$ implies $(A_i \cap B) \subseteq H$, contrary to assumption.  Therefore, condition (2) holds.  

Finally, consider the inclusions
 \[
  (A_{i+1} \cap B) \subseteq A_{i+1} \subseteq A_i \text{ and } (A_{i+1} \cap B) \subseteq A_{i+1} \subseteq H.
 \]
 Since $[A_i : A_i \cap B] \ge \aleph_0$, $[A_i : A_{i+1} \cap B] \ge \aleph_0$.  Moreover, since $A_i \cap B \nsim H$, $[H : A_{i+1} \cap B] \ge \aleph_0$.  However, by Proposition \ref{Prop_dpMinChar}, at least one of $[H : A_{i+1}]$ and $[A_i : A_{i+1}]$ is finite, as either $H \precsim A_i$ or $A_i \precsim H$.  Therefore, from
\begin{align*}
[A_i : A_{i+1} \cap B] = [A_i : A_{i+1}][A_{i+1} : A_{i+1} \cap B] & \ge \aleph_0\\
[H : A_{i+1} \cap B] = [H : A_{i+1}][A_{i+1} : A_{i+1} \cap B] & \ge \aleph_0
\end{align*}
we obtain $[A_{i+1} : A_{i+1} \cap B] \ge \aleph_0$.  Hence, condition (3) holds.  This completes the construction, showing that $T$ is not convexly orderable.
\end{proof}

Before turning to the classification of VC-minimal abelian groups, we will need two lemmas. Both address the question of transferring VC-minimality between an abelian group and its direct summands. For groups $\fA = (A; +)$ and $\fB = (B; +)$, let $\fA \oplus \fB = (A \oplus B; +)$.

\begin{lem}\label{Lem_OtherDirectionOK}
 If $\fB$ is any abelian group and $\Th(\fA \oplus \fB)$ is VC-minimal, then $\Th(\fA)$ is VC-minimal.
\end{lem}

\begin{proof}
Assume $T^* = \Th(\fA \oplus \fB)$ is VC-minimal.  By Theorem \ref{Thm_CharofVCMinAG} (3), $T^*$ is dp-minimal and $A \oplus B$ is upwardly coherent.  By the proof of Lemma \ref{Lem_UpwardCohereVCMin}, there exists $\mathcal{H} \subseteq \PP(A \oplus B)$ such that $(\mathcal{H}; \subseteq)$ is a linear order and, for all $k$ and $m$, $\phi_{k,m}(\fA \oplus \fB)$ is a finite union of cosets of some $H_{k,m} \in \mathcal{H}$. Thus, we may write
 \[
  \phi_{k,m}(\fA \oplus \fB) = \bigcup_{i \le n} (a_i\oplus b_i)+H_{k,m}
 \]
 for some choice of $a_i \in A$, $b_i \in B$.

If $\pi_A$ denotes the projection of $\fA\oplus\fB$ onto $\fA$, note that $\pi_A(\phi_{k,m}(\fA\oplus\fB))=\phi_{k,m}(\fA)$. So, clearly, $\mathcal{H}_A=\pi_A(\mathcal{H})$ is also linearly ordered by $\subseteq$.  Moreover, we have 
\[
  \phi_{k,m}(\fA) = \bigcup_{i \le n} a_i + \pi_A(H_{k,m}).
 \]
Therefore, using $\mathcal{H}_A$ in Lemma \ref{Lem_VCMinAbGrps}, we see that $T = \Th(\fA)$ is VC-minimal.
\end{proof}

\begin{lem}\label{Lem_FiniteDoesntMatter}
 If $\fB$ is a finite abelian group, then $\Th(\fA)$ is VC-minimal if and only if $\Th(\fA \oplus \fB)$ is VC-minimal.
\end{lem}

\begin{proof}
 Suppose $T = \Th(\fA)$ is VC-minimal.  Again recalling the proof of Lemma \ref{Lem_UpwardCohereVCMin}, there exists $\mathcal{H} \subseteq \PP(A)$ so that $(\mathcal{H}; \subseteq)$ is a chain and, for all $k$ and $m$, $\phi_{k,m}(\fA)$ is a finite union of cosets of some $H_{k,m}\in\mathcal{H}$.  For each $H \in \mathcal{H}$, choose a subgroup $B(H) \subseteq B$ minimal (with respect to $\subseteq$) such that $H \oplus B(H) \in \PP(A \oplus B)$.  Finally, let
 \[
  \mathcal{H}^* = \set{H \oplus B(H)}{H \in \mathcal{H}}.
 \]
 We verify that $\mathcal{H}^*$ satisfies the hypotheses of Lemma \ref{Lem_VCMinAbGrps} for $\fA \oplus \fB$.
 
 First, to see that $\mathcal{H}^*$ is a linear order under $\subseteq$, suppose $H_1 \subseteq H_2$ from $\mathcal{H}$. As  
\[
(H_1 \oplus B(H_1)) \cap (H_2 \oplus B(H_2)) = H_1 \oplus (B(H_1) \cap B(H_2))
\] 
is again an element of $\PP(A \oplus B)$, the minimality of $B(H_1)$ implies $B(H_1)=B(H_1) \cap B(H_2)$. Thus $B(H_1)\subseteq B(H_2)$ and $H_1\oplus B(H_1) \subseteq H_2\oplus B(H_2)$.
 
Second, we wish to show that $\phi_{k,m}(\fA \oplus \fB)$ is a boolean combination of cosets of elements of $\mathcal{H}^*$. Since we already know that $\phi_{k,m}(\fA)$ is a finite union of cosets of $H_{k,m}$, and $B$ is finite, it suffices to show that
 \[
  H_{k,m} \oplus B(H_{k,m}) \subseteq \phi_{k,m}(\fA) \oplus \phi_{k,m}(\fB) = \phi_{k,m}(\fA \oplus \fB).
 \]
 That is, we need to show $B(H_{k,m}) \subseteq \phi_{k,m}(\fB)$.  If not, however,
\[
\left(H_{k,m} \oplus B(H_{k,m})\right) \cap \phi_{k,m}(\fA \oplus \fB) = H_{k,m} \oplus \left(B(H_{k,m}) \cap \phi_{k,m}(\fB)\right)
\]
would be in $\PP(A \oplus B)$, in which case $B(H_{k,m}) \cap \phi_{k,m}(\fB)$ would contradict the minimality of $B(H_{k,m})$.
 
Therefore, $\mathcal{H}^*$ satisfies the conditions of Lemma \ref{Lem_VCMinAbGrps}, proving VC-minimality of $\Th(\fA\oplus\fB)$. The converse follows immediately from Lemma \ref{Lem_OtherDirectionOK}.
\end{proof}

We are now ready to prove an analog to Proposition \ref{Prop_ADHMS527} for VC-minimal (and convexly orderable) theories of abelian groups. The proposition gives a strong starting point, a complete list of dp-minimal theories of abelian groups. Theorem \ref{Thm_CharofVCMinAG} and the above lemmas provide a set of tools for determining which of these are VC-minimal.

\begin{thm}\label{Prop_VCMinAG}
$T$ is VC-minimal (and convexly orderable) if and only if $\fA$ is elementarily equivalent to one of the following abelian groups:
 \begin{enumerate}
  \item 
  $\bigoplus\limits_{p \text{ prime}} \left( \mathbb{Z}\left( p^\infty \right)^{(\beta_p)} \right) \oplus \left( \mathbb{Z}_{(q)} \right)^{(\gamma)} \oplus \mathbb{Q}^{(\delta)} \oplus B$ for a fixed prime $q$, finite abelian group $B$, and cardinals $\beta_p$, $\gamma$, and $\delta$ such that $\beta_p < \aleph_0$ for all $p \neq q$;
  \item 
  $\bigoplus\limits_{p \text{ prime}} 
  \left( B_p \oplus
    \mathbb{Z}\left( p^\infty \right)^{(\beta_p)} \oplus 
    \left( \mathbb{Z}_{(p)} \right)^{(\gamma_p)}\right) \oplus 
    \mathbb{Q}^{(\delta)}$ for a fixed prime $q$, finite $p$-groups $B_p$, and cardinals $\beta_p$, $\gamma_p$, and $\delta$ such that 
    $\beta_p < \aleph_0$ for all $p \neq q$ and
     $\gamma_p < \aleph_0$ for all $p$ (incuding $q$);
  \item $\left( \mathbb{Z} / p^k \mathbb{Z} \right)^{(\alpha)} \oplus \left( \mathbb{Z} / p^{k+1} \mathbb{Z} \right)^{(\beta)} \oplus B$ for some prime $p$, $k \ge 1$, finite abelian group $B$, and cardinals $\alpha$ and $\beta$, at least one of which is infinite.
 \end{enumerate}
\end{thm}

\begin{proof}
 Suppose $T$ is dp-minimal.  By Proposition \ref{Prop_ADHMS527}, it falls under one of two categories.  $T$ is either the theory of a group as in (3) above; or, $\fA$ is elementarily equivalent to
 \begin{equation}\label{eqn5141}
\bigoplus\limits_{i\ge 1} \left(\mathbb{Z} / p^i \mathbb{Z} \right)^{(\alpha_{p.i})} \oplus
 \mathbb{Z}\left( p^\infty \right)^{(\beta_p)} \oplus 
 \left( \mathbb{Z}_{(p)} \right)^{(\gamma_p)} \oplus B
 \tag{$\star$}
 \end{equation}
for a prime $p$, nonsingular abelian group $B$, and cardinals $\alpha_{p,i}$, $\beta_p$, and $\gamma_p$ with each $\alpha_{p,i}$ finite. 

For the former category, it follows from Corollary \ref{Cor_OtherVCMin} and Lemma \ref{Lem_FiniteDoesntMatter} that the group in (3) is also VC-minimal.
 
For the latter, first recall that by results of Szmielew \cite{szm}, any abelian group is elementarily equivalent to one of the form
\[
\bigoplus\limits_{p \text{ prime}} 
  \left(
  \bigoplus\limits_{i\ge 1}\left(\mathbb{Z}/p^i\mathbb{Z}\right)^{(\alpha_{p,i})} \oplus
    \mathbb{Z}\left( p^\infty \right)^{(\beta_p)} \oplus 
    \left( \mathbb{Z}_{(p)} \right)^{(\gamma_p)}\right) \oplus 
    \mathbb{Q}^{(\delta)}.
\]
It is straightforward to verify that such a group is only nonsingular if each $\alpha_{p,i}$, $\beta_p$, $\gamma_p$, and $\set{i}{\alpha_{p,i}>0}$ is finite. For instance, for $B=\left(\mathbb{Z}_{(p)}\right)^{(\gamma_p)}$, we have $B/pB=\left(\mathbb{Z}/p\mathbb{Z}\right)^{(\gamma_p)}$, which is finite iff $\gamma_p$ is. Hence, (\ref{eqn5141}) becomes
\begin{equation}\label{eqn5142}
\bigoplus\limits_{p \text{ prime}} 
  \left( \bigoplus\limits_{i\ge 1}\left(\mathbb{Z}/p^i\mathbb{Z}\right)^{(\alpha_{p,i})} \oplus
    \mathbb{Z}\left( p^\infty \right)^{(\beta_p)} \oplus 
    \left( \mathbb{Z}_{(p)} \right)^{(\gamma_p)}\right) \oplus 
    \mathbb{Q}^{(\delta)}
\tag{$\dagger$}
\end{equation}
with each $\alpha_{p,i}$ finite and $\beta_p$, $\gamma_p$, and $\set{i}{\alpha_{p,i}>0}$ finite for $p\ne q$. In other words, writing $B_p=\bigoplus\limits_{i\ge 1}\left(\mathbb{Z}/p^i\mathbb{Z}\right)^{(\alpha_{p,i})}$, we have that $B_p$ is a finite $p$-group for all $p\ne q$.

Suppose, then, that (\ref{eqn5142}) is VC-minimal. We show that (\ref{eqn5142}) is as in (1) or (2). By Corollary \ref{Cor_ExamplesNotVCMinAG} and Lemma \ref{Lem_OtherDirectionOK}, $B_q$ must also be finite. If $\gamma_q<\aleph_0$, then we are in case (2). 

Thus, suppose that $\gamma_p\ge\aleph_0$. Notice that $qA \precnsim A$.
We must show that $B=\bigoplus_pB_p$ is finite and $\gamma_p=0$ for $p\ne q$.

If $\gamma_p>0$ for some $p\ne q$, then $qA\precnsim p^nA$ for all $n$. However, there is no $H\in\PP(A)$ with $H\sim qA$ such that $H\subseteq p^nA$ for all $n$. Therefore, the $\sim$-class of $qA$ is not upwardly coherent, contradicting Theorem \ref{Thm_CharofVCMinAG}.

If $B_p$ is nonzero for infinitely many primes $p$, let $p_0,p_1,\ldots$ enumerate all such primes, excluding $q$. Then we have
\[
qA\precnsim \left(\prod\limits_{i\le n}p_i\right)A.
\]
But there is no $H\in\PP(A)$ with $H\sim qA$ such that $H\subseteq \left(\prod_{i\le n}p_n\right) A$ for every $n$, again contradicting upward coherence of the $\sim$-class of $qA$.

We have thus established that the theory of a VC-minimal abelian group belongs to one of the cases (1), (2), or (3). It remains only to show that the groups in (1) and (2) are indeed VC-minimal.

For both cases, $A[q^n]$ witnesses the upward coherence of its $\sim$-class for every $n$.  In case (2), $k A \sim A$ for all $k$, so the chain of $\PPt(A)$ is given by
 \[
  0 \precnsim A[q] \precnsim A[q^2] \precnsim ... \precnsim A,
 \]
 and each $\sim$-class is upwardly coherent. Furthermore, each group in $\PP(A)$ is a boolean combination of cosets of groups in this chain.  The details of this computation can be found in Lemma 5.28 of \cite{ADHMS}.

  In case (1), in addition to $A[q^n]$, we also have that $q^n A$ witnesses the upward coherence of its $\sim$-class.  The chain of $\PPt(A)$ is given by
 \[
  0 \precnsim A[q] \precnsim A[q^2] \precnsim ... \precnsim q^2A \precnsim qA \precnsim A.
 \]
Again, we refer to Lemma 5.28 of \cite{ADHMS} to see that the groups in this chain generate every member of $\PP(A)$. 

In both cases, therefore, $A$ is upwardly coherent.  By Theorem \ref{Thm_CharofVCMinAG}, $T$ is VC-minimal. 
 \end{proof}
 
 \subsection*{Acknowledgments} We would like to express our thanks to the referee for pointing out an error in the original draft of Proposition \ref{Prop_ADHMS527} and consequently, Theorem \ref{Prop_VCMinAG}.

\end{document}